\documentclass[12pt,letterpaper]{article}
\usepackage{amsmath}
\usepackage{amsfonts}
\usepackage{amsthm}
\usepackage{amssymb}
\usepackage{array}
\usepackage{caption}
\usepackage{color}
\usepackage{enumitem}
\usepackage{epic,eepic}
\usepackage{graphics}
\usepackage{graphicx}
\usepackage{hyperref}
\usepackage{mathtools}
\usepackage{rotating}
\usepackage{verbatim}
\newtheorem{prop}{Proposition}

\newtheorem{theorem}[prop]{Theorem}
\newtheorem{lemma}[prop]{Lemma}

\begin{document}
\title{Cookie Monster Devours Naccis}
\author{Leigh Marie Braswell\\Phillips Exeter Academy \and Tanya Khovanova\\MIT}
\maketitle

\begin{abstract}
In 2002, Cookie Monster appeared in \textit{The Inquisitive Problem Solver}. The hungry monster wants to empty a set of jars filled with various numbers of cookies. On each of his moves, he may choose any subset of jars and take the same number of cookies from each of those jars. 

The \textit{Cookie Monster number} is the minimum number of moves Cookie Monster must use to empty all of the jars. This number depends on the initial distribution of cookies in the jars.

We discuss bounds of the Cookie Monster number and explicitly find the Cookie Monster number for Fibonacci, Tribonacci and other nacci sequences.
\end{abstract}

\section{Cookie Monster wants cookies}

Cookie Monster's Mommy has a set of $k$ jars filled with  cookies. In one move she allows Cookie Monster to choose any subset of jars and take (or devour immediately) the same number of cookies from each of those jars. Cookie Monster always wants to empty all of the jars in as few moves as possible. 

The first time Cookie Monster's Mommy set up this system, she made the mistake of putting the same number of cookies in each of the jars. Unfortunately, Cookie Monster finished emptying the jars in one move. Since then, his Mommy has became smarter and always puts distinct numbers of cookies in the jars.

Cookie Monster learned in kindergarten that it is a good idea to give notation to numbers. So he decided to denote the set of $k$ jars with $j_{1} < j_{2} < \cdots < j_{k}$ cookies as a set $S = \{j_{1}, j_{2}, \ldots, j_{k} \}$.

The next time Cookie Monster's Mommy filled the jars, she put 1 cookie in the first jar, 2 cookies in the second jar and so on: $i$ cookies in the $i$-th jar. She was hoping that because the number of cookies is different in each jar, Cookie Monster would empty one jar at a time, and it would take him $k$ moves to empty all the jars. Little did she know how smart her son was. Though she filled 7 jars with cookies: $S=\{1,2,3,4,5,6,7\}$, Cookie Monster emptied all the jars not in 7, but rather in 3 moves. On the first move Cookie Monster took 4 cookies from all of the  jars he could. That is, he took 4 cookies from the 4 jars with four or more cookies and consumed a total of 16 cookies. After that, the set of numbers of cookies in jars became $\{1, 2, 3,0,1,2,3 \}$. On his next move he took 2 cookies from the jars with two or more cookies, so the set of numbers of cookies became $\{1, 0, 1,0,1,0,1 \}$. 

Cookie Monster noticed that when the numbers of cookies in two jars are the same, he can treat the jars in the same way. So for the purposes of inventing his strategy, he only needs to know the set of numbers of distinct cookies in jars. When describing the number of cookies in jars, he can ignore duplicates and empty jars. Cookie Monster also noticed that it is good to reduce the number of cookies in any jar to a number of cookies in a jar that he didn't touch. For the purposes of getting rid of the jars, it is as good as emptying the jar. On the second thought, it is not as good as emptying the jar, since he gets to eat fewer cookies from this jar on that particular move. Anyway, he decided to call a jar that is emptied or reduced to an existing number a \textit{discarded} jar. So in his brilliant strategy above, on the first move he discarded 4 jars and reduced the set of cookie numbers to $\{1, 2, 3\}$. On the second move he discarded two more numbers from the set and reduced it to just $\{1\}$.

\section{Cookie Monster number}

After spending so much time thinking about it, Cookie Monster decided that he needs a name for the optimal number of moves for a particular set. The minimum number of moves Cookie Monster must use to empty all of the jars represented by the set $S$ is now called the \emph{Cookie Monster number of $S$}, or in short, $\text{CM}(S)$.

He was also worried about the last time: was there a way to empty the jars in less than three moves? 

He was very gifted and finally came up with the theorem:

\begin{theorem}
Cookie Monster needs at least $\lceil \log_2(k+1)\rceil $ steps to empty any $k$ distinct jars. 
\end{theorem}

\begin{proof}
Suppose that $k$ jars contain distinct numbers of cookies, and let $f(k)$ be the number of distinct non-empty jars after the first move of Cookie Monster. He claims that $k \leq 2f(k)+1$. Indeed, after the first move, there will be at least $k-1$ non-empty jars, but there cannot be three identical non-empty jars. That means, the number of jars plus 1 can't decrease faster than twice each time.
\end{proof}

Cookie Monster's Mommy was pleasantly surprised to see how smart her son was, but she was worried that he finished the jars too fast. For the next time, she decided to think things through. If she has 7 jars, is there a way to place the cookies so that Cookie Monster would need 7 moves?
 She decided to place $2^i$ cookies in the $i$-th jar: $j_i = 2^i$.

Cookie Monster tried very hard to think of a strategy to empty the $k$ jars in less than $k$ moves. He even invented the following lemma:

\begin{lemma}\label{thm:commutativity}
Cookie Monster may perform his optimal moves in any order.
\end{lemma}

Finally he understood. Here is his new theorem:

\begin{theorem}
If the number of cookies in the $i$-th jar is greater than the total number of cookies in jars 1 through $i-1$ for any $i$, there is no way to empty all the jars faster than one by one.
\end{theorem}

\begin{proof}
The largest jar has more than the total number of cookies in all of the other jars combined. Therefore, any strategy has to include a step in which Cookie Monster only takes cookies from the largest jar. He will not jeopardize the strategy if he takes all the cookies from the largest jar in the first move. Applying the induction process, Cookie Monster sadly sees that he needs $k$ steps for $k$ jars.
\end{proof}

Yes, Cookie Monster remembers that $2^i > 1+2+\dots +2^{i-1}$. His mother didn't give him a chance to empty the jars faster than the boring one by one strategy does! He starts crying. His mother comforts him and promises in the future to place the cookies in such a way that he can always empty the jars faster than the number of jars.

\section{Yummy Fibonacci sequence}

The next time she decides to challenge her son, Cookie Monster's Mommy remembered that Cookie Monster likes the Fibonacci sequence. The Fibonacci sequence is defined as $F_{0} = 0$, $F_{1} = 1$, and $F_{i} = F_{i-2} + F_{i-1}$ for $i \geq 2$. She didn't want to have an empty jar and two jars with the same number of cookies, so she decided that the smallest jar will contain $F_2$ cookies, and so on: the $i$-th jar contains $F_{i+1}$ cookies and the set of $k$ jars contains $\{ F_2, F_3, \ldots, F_{k+1}\}$ cookies.

On his first move Cookie Monster took $F_{k}$ cookies from jars numbered $k-1$ and $k$. He emptied the jar $k-1$. And the $k$-th jar now has $F_{k+1} - F_k = F_{k-1}$ cookies, which is the same number of cookies that jar $F_{k-2}$ has. Thus, he discarded two jars in one move. So he can proceed like this and empty the jars in $k/2$ moves. Hooray! 

Wait a minute! Something is off. You need to be careful with cookies nowadays. If $k$ is odd, $k/2$ is not an integer, so Cookie Monster needs an extra move to empty the last jar, and the total number of moves is $(k+1)/2$. If $k$ is even, he still can't empty the jars in $k/2$ moves, because he can't empty the last two jars in one move. He needs one more move, so the answer is $k/2+1$. He can combine both answers for any $k$ to be: $\lceil (k+1)/2 \rceil$.

Cookie Monster was happy that he finished all the jars in almost half the time. But then he started wondering. What if there was a faster way? He needed to prove for himself that there was not.

\begin{theorem}
For $k$ jars with the set of cookies $S = \{F_{2}, \dots, F_{k+1}\}$, the Cookie Monster number is $\text{CM}(S) = \lceil \frac{k+1}{2} \rceil$.
\end{theorem}

\begin{proof}
Cookie Monster remembered that one of the most famous identities for the Fibonacci numbers is the following: $F_{k+1} -1 = \Sigma_{i=1}^{k-1}F_i$. That means that there must be a move that touches the $k$-th jar and doesn't involve the first $k-2$ jars. He can't be less optimal if he starts with this move and discards both the largest and the second largest jars. That means his strategy must be optimal, or, more precisely, it must be one of the optimal strategies.
\end{proof}

\section{Tribonacci}

Cookie Monster's Mommy decides to find a more difficult sequence. She googled around and discovered Tribonacci numbers. Tribonacci numbers are similar to Fibonacci numbers, but the next term is the sum of the previous three terms rather than the previous two terms. Specifically, you start with $T_0 = 0$, $T_1 = 0$, and $T_2=1$, then use the recursion: $T_k = T_{k-1} + T_{k-2} + T_{k-3}$. You get the following sequence: 0, 0, 1, 1, 2, 4, 7, 13, 24, 44, $\ldots$. To avoid empty jars and duplicates, Cookie Monster's Mommy puts the following set in her $k$ jars: $S = \{T_3, T_4, \ldots, T_{k+2} \}$.

The main property of the Fibonacci sequence is that the next term is the sum of previous terms. Cookie Monster is smart and he uses exactly this fact to make a strategy for emptying jars with Tribonaccis. He thinks back to the jars containing $F_2, F_3, \ldots, F_{k+1}$ cookies. He remembers taking $F_{k}$ cookies from the two largest jars; in doing so, he emptied the jar with $F_{k}$ cookies and reduced the number of cookies in the jar with $F_{k+1}$ cookies to $F_{k-1}$ cookies, discarding both jars. How can he extend this method to Tribonaccis? Can he discard two jars in one move? Not so fast. If he takes $T_{k+1}$ cookies from the two largest jars, he empties the second largest jar, but the number of cookies in the largest jar is $T_{k+2} - T_{k+1}$: still too big to equal a number in another jar. But --- AHA --- he can also take $T_{k}$ cookies from the third largest and the largest jars, emptying the third largest jar and reducing the largest jar to $T_{k+2} - T_{k+1} - T_k = T_{k-1}$. Now the largest jar is discarded, as the number of cookies left there is the same as the number of cookies in another jar. 

So Cookie Monster can empty three jars in two moves, which means that he empties all the jars in about $2k/3$ moves. Now he remembers to check the endgame to get a precise number. If $k$ has remainder 1 modulo 3, he needs one more move for the last jar: $2\lfloor k/3 \rfloor + 1$. If $k$ has remainder 2 modulo 3, he needs two more moves: $2\lfloor k/3 \rfloor + 2$. If $k$ has remainder 0 modulo 3, he needs three moves for the last group of three: $2\lfloor k/3 \rfloor + 1$. He can combine his calculations into the formula for any $k$: $\lfloor 2k/3 \rfloor + 1$. Now what is left is to prove is that this is indeed his famous number.

\begin{theorem}
When $k$ jars contain a set of Tribonacci numbers $S = \{T_{3}, \dots, T_{k+2}\}$, then $\text{CM}(S) = \lfloor \frac{2k}{3} \rfloor + 1$.
\end{theorem}

Where to start this proof? Cookie Monster remembers that he used inequalities and identities to prove previous theorems. The Tribonacci sequence grows faster than the Fibonacci sequence, so the following must be true: $T_{k+1} > \Sigma_{i=1}^{k-1}T_i$. But this might not be enough. Cookie Monster looked into the sky and came up with the second identity he needed:

\begin{lemma}
The Tribonacci sequence satisfies the inequalities:
\begin{itemize}
\item $T_{k+1} > \Sigma_{i=1}^{k-1}T_i$
\item $T_{k+2} - T_{k+1} > \Sigma_{i=1}^{k-1}T_i$.
\end{itemize}
\end{lemma}

\begin{proof}
Let's prove this by induction. For the base case $k = 1$, both inequalities are true. Suppose they are both true for some $k$. If $T_{k}$ is added to both sides of both inequalities' induction hypotheses, we get 
\begin{eqnarray*}
T_{k+1} + T_k > \Sigma_{i=1}^{k}T_i
\end{eqnarray*}
\begin{eqnarray*}
T_{k+2} - T_{k+1} + T_k > \Sigma_{i=1}^{k}T_i.
\end{eqnarray*}
But $T_{k+1} + T_k = T_{k+3} - T_{k+2}$ and $T_{k+2} - T_{k+1} + T_k < T_{k+3}$. Therefore, they are both true for $k+1$. Cookie Monster is surprised that the inequalities are intertwined: he can use one to prove the other and vice versa.
\end{proof}

Now Cookie Monster needs to prove the theorem.

\begin{proof}
Consider the largest three jars. The largest jar and the second largest jar each have more cookies than the remaining $k-3$ jars. That means there should be a move that includes the second largest and largest jars that doesn't touch the smallest $k-3$ jars. Because of the second inequality, and because we do not actually have $T_1$ and $T_2$ jars in the set, the largest jar needs one more move that doesn't touch the smallest $k-3$ jars to be discarded. As we need at least two moves to touch and discard the last three jars, discarding all three jars in two moves is the best we can do.
\end{proof}

\section{``Nacci" sounds Italian}

The reader might wonder what ``nacci" in the title might mean. It sounds a little bit like gnocchi, Italian dumplings. Sure, Cookie Monster likes gnoccis too, but he likes cookies better. And nacci sounds Italian because Fibonacci is from Italy.

Anyway, Cookie Monster's Mommy was happy that the Tribonacci sequence took Cookie Monster more moves than Fibonacci sequence. Is there a sequence that will take even more moves, but allows her to keep her promise? The next such sequence should probably be called Tetrabonacci. She googled it and found that there are no such numbers. But there is a Tetranacci sequence whose terms are sums of the previous four terms.  Why are mathematicians so illogical?

She realized that because her son knows addition and understands mathematics, he would figure out that the number of moves to consume Tetranaccis is about $3k/4$. Can she continue her trend forever? Can she find other sequences that require an increasing number of moves per jar, but still fewer moves than the total number of jars? She found Pentanacci numbers. Isn't Pentabonacci a more interesting name? Then there are Hexanacci numbers and Heptanacci numbers. The next sequence is Octanacci numbers, followed by Nonacci numbers.

Yes, she can torture her son forever. However, people ran out of unique names, so they call these sequences $n$-nacci series, and there is an infinite number of these sequences. The Cookie Monster number of the $n$-nacci series is of the order $\frac{(n-1)k}{n}$. It gets closer and closer to $k$, but it never reaches it.

Cookie Monster denotes the $n$-nacci sequence as $n_i$. Theoretically, we need to specify the first $n$ terms of an $n$-nacci sequence. For example, we can start a Tetranacci-like sequence with four zeros. Then the sequence becomes all zeros. This would be a very sad sequence as all the jars would be empty, and there would be a hungry and irritated Cookie Monster. So mathematicians agree to define nacci sequences as the lexicographically smallest nacci-like sequences. To put it in English, the first $n$ terms of an $n$-nacci sequence are: $n_0 = n_1 = \cdots = n_{n-2} = 0$ and $n_{n-1} = 1$. So the Tetranacci sequence starts as  	0, 0, 0, 1, 1, 2, 4, 8, 15, 29, 56, 108.

To prevent duplicate numbers of cookies, Cookie Monster's Mommy starts the sequence with $n_{n}$, so the set of $k$ jars is $\{n_{n}, \dots, n_{n+k-1}\}$.

Here is Cookie Monster's strategy for dealing with $n$-nacci sequences, which he calls \textit{cookie-monster-knows-addition}: He takes $n - 1$ moves to remove $j_{k-i}$ cookies from the $k$-th and $k-i$-th largest jars for each $i$ such that $0 < i < n $. In doing this, $n$ jars are emptied in $n - 1$ moves. This process can be repeated, until at most $n - 1$ elements remain, which he empties one by one.  Thus, when $S = \{n_{n}, \dots, n_{n+k-1}\}$, the Cookie Monster number $\text{CM}(S) = \lceil \frac{n-1}{n}(k-1) \rceil +1$.

Yes, Cookie Monster figured out how to consume naccis. 

Cookie Monster: ``Me want cookie!"

\section{Bibliography and Acknowledgements}\label{sec:acknowledgements}

Cookie Monster is proud that people study his cookie eating strategies. The problem first appeared in \textit{The Inquisitive Problem Solver} by Vaderlind Guy and Larson \cite{VGL}. The research was continued by Belzner \cite{B}, Bernardi and Khovanova \cite{BK} and Cavers \cite{C}. This paper contains the original research by the authors.

Cookie Monster and the authors are grateful to MIT-PRIMES program for supporting this research.


\begin{thebibliography}{9}


\bibitem{B} M.~Belzner, Emptying Sets: The Cookie Monster Problem, math.CO arXiv:1304.7508, (2013).
\bibitem{BK} O. Bernardi and T.~Khovanova, The Cookie Monster Problem, available at: \url{http://blog.tanyakhovanova.com/?p=325}, (2011).
\bibitem{VGL} P. Vaderlind, R.K. Guy, and L.C. Larson. \textit{The Inquisitive Problem Solver.} The Mathematical Association of America (2002).
\bibitem{C} M.~Cavers, Cookie Monster Problem Notes. University of Calgary Discrete Math Seminar, private communication (2011).

\end{thebibliography}
\end{document}